\documentclass{amsart}

\usepackage{amsmath,amsthm,amsfonts,amssymb}

\usepackage{enumerate}  
\usepackage{hyperref}  
\hypersetup{colorlinks=true,linktoc=none, linkcolor=blue,citecolor=blue}

\usepackage{lmodern}
\usepackage[T1]{fontenc}

\usepackage{color}
\usepackage{mathtools}
\usepackage{nicefrac}
\mathtoolsset{showonlyrefs} 
\usepackage{mathrsfs}

\newtheorem{thm}{Theorem}[section]
\newtheorem{cor}[thm]{Corollary}
\newtheorem{lem}[thm]{Lemma}
\newtheorem{prop}[thm]{Proposition}
\theoremstyle{definition}
\newtheorem{defn}[thm]{Definition}
\theoremstyle{remark}
\newtheorem{rem}[thm]{Remark}

\DeclareMathOperator*{\supp}{supp}

\newcommand{\R}{\mathbb{R}}  
\newcommand{\N}{\mathbb{N}}  
\newcommand{\h}{\mathcal{H}}  
\newcommand{\C}{\mathcal{C}}  
\newcommand{\di}{\mathrm{d}}  
\newcommand{\ve}{\varepsilon}  

\numberwithin{equation}{section}

\begin{document}

\title[Optimization problem]{An optimization problem for the first eigenvalue of the $p-$fractional Laplacian}

\author[L. Del Pezzo, J. Fern\'andez Bonder and L. L\'opez R\'ios]{Leandro Del Pezzo, Juli\'an Fern\'andez Bonder and Luis L\'opez R\'ios}

\address{Departamento de Matem\'atica FCEN - Universidad de Buenos Aires and IMAS - CONICET. Ciudad Universitaria, Pabell\'on I (C1428EGA)
Av. Cantilo 2160. Buenos Aires, Argentina.}

\email[J. Fernandez Bonder]{jfbonder@dm.uba.ar}
\urladdr{http://mate.dm.uba.ar/~jfbonder}

\email[L. Del Pezzo]{ldpezzo@dm.uba.ar}
\urladdr{http://cms.dm.uba.ar/Members/ldpezzo/}

\email[L. L\'opez R\'ios]{llopez@dm.uba.ar}

\subjclass[2010]{35P30, 35J92, 49R05}

\keywords{Optimization, Fractional Laplacian, Nonlinear eigenvalues}

\begin{abstract}
In this paper we analyze an eigenvalue problem related to the nonlocal $p-$Laplace operator plus a potential. After reviewing some elementary properties of the first eigenvalue of these operators (existence, positivity of associated eigenfunctions, simplicity and isolation) we investigate the dependence of the first eigenvalue on the potential function and establish the existence of some {\em optimal} potentials in some admissible classes.
\end{abstract}

\maketitle

\section{Introduction}\label{intro}

In this paper we study the following non-linear non-local eigenvalue problem
\begin{equation} \label{eq:autovalores}
  \left\lbrace 
	\begin{aligned}
		(-\Delta_{p})^s u+ V(x) |u|^{p-2} u &= \lambda |u|^{p-2}u && 
		\text{in } \Omega, \\
		u &= 0 &&\text{in } \mathbb{R}^n\setminus\Omega,
	\end{aligned}
  \right.
\end{equation}
where $\Omega\subset\mathbb{R}^n$, $n\geq 1$, is a smooth bounded domain, $0<s<1 < p <\infty$, and $\lambda \in \R$. The potential $V$ is in $L^q(\Omega)$, $\max\{1,\tfrac{n}{sp}\}<q<\infty$, and $(-\Delta_p)^s$ is the fractional $p$-Laplacian operator, which, in a suitable regularity class (see \cite{IMS, KKL}), is given by  
\begin{equation}\label{eq:splap}
(-\Delta_p)^s u(x) := \text{p.v.} \int_{\R^n} \frac{|u(x)-u(y)|^{p-2}(u(x)-u(y))}{|x-y|^{n+sp}}\, \di y. 
\end{equation}

Observe that, in the case $p=2,$ $(-\Delta_2)^s=(-\Delta)^s$ is the usual fractional Laplace operator.

First, we devoted the paper to the study of problem \eqref{eq:autovalores}. For this eigenvalue problem we prove the existence of a first eigenvalue and then analyze properties of the associated eigenfunction.

Once the existence of this first eigenvalue is established we arrive at the main point of this article, that is the optimization of this first eigenvalue with respect to the potential function $V$.

This type of problems appears naturally in the study of the fractional Shr\"odinger equation. The eigenvalues and eigenfunctions of \eqref{eq:autovalores} are the associated fundamental states of the system. This is of particular interest in the case $p=2$. See \cite{Laskin}. We want to stress that all the results in this paper are new even in the linear case that corresponds to $p=2$. 

The problem that we want to address is the following. Suppose that we know that the potential $V$ possesses some bound (say $\|V\|_q\le M$), then what can be said about the fundamental state of the system? That is, if we only know the information $\|V\|_q$ for some $q>1$, then what bounds can we have for the first eigenvalue of \eqref{eq:autovalores} and what information can we deduce for the associated eigenfunction. 

In the classical linear setting, that is when $p=2$ and when the fractional Laplacian is replaced by the standard Laplacian operator, this problem was first studied in \cite{Ashbaugh} and then extended to the $p-$Laplacian operator in \cite{FBDP}.

As far as we know, no investigation was done so far in the fractional setting.

\subsection*{Organization of the paper}
After this short introduction, we include a section (Section 2) where some preliminaries on fractional Sobolev spaces that are used throughout the paper are collected.

In Section 3 we analyze problem \eqref{eq:autovalores} and show the existence of a first eigenvalue, together with a nonnegative associated eigenfunction. Moreover, we show the simplicity and isolation of this eigenvalue.

In Section 4, we study some properties about the dependence of the principal eigenvalue on the potential function $V$.

Finally, in Section 5, we prove the main results of the paper that is the study of the optimization problem for \eqref{eq:autovalores} where $V$ is restricted to belong to some ball in $L^q$.

	
\section{Preliminaries}\label{Prel}

\subsection{Fractional spaces}

Let us recall some well known facts about fractional spaces. Among the many references in this subject, let us mention \cite{Adams, DD, Grisvard}, which are enough for our purposes. Also, the excellent review article \cite{DNPV} will cover anything that is needed here. Throughout this section we consider $0<s<1$ and $1 < p < \infty$ to be fixed. Given an open set $\Omega \subset \R^n$, the fractional Sobolev space $W^{s,p}(\Omega)$ is defined by
\[
	W^{s,p}(\Omega) = \left\{ u\in L^p(\Omega)\colon
	\dfrac{u(x)-u(y)}{|x-y|^{\frac{n}{p}+s}}\in L^p(\Omega\times\Omega)\right\}.
\]
This space is endowed with the norm
\[
	\|u\|_{s,p;\Omega} \coloneqq \|u\|_{W^{s,p}(\Omega)} = 
	\left(\|u\|_{p;\Omega}^p + [u]_{s,p,\Omega}^p
	\right)^{\tfrac1p},
\]
where
\[
	\|u\|_{p;\Omega} \coloneqq \|u\|_{L^p(\Omega)}  = \left( \int_\Omega 
	|u|^p\, \di x \right) ^{\tfrac1p}
\]
and 
\[
	[u]_{s,p;\Omega} \coloneqq \left( 
	\iint_{\Omega\times \Omega}\dfrac{|u(x)-u(y)|^p}{|x-y|^{n+sp}}\,\di x \di y \right) ^ {\tfrac 1p}
\]
is called the Gagliardo seminorm. If $\Omega = \R^n$, we shall omit the set in the notation:
\[
  \|u\|_{s,p} \coloneqq \|u\|_{s,p;\R^n}, \quad  \|u\|_p \coloneqq \|u\|_{p;\R^n}    \quad \text{and} \quad [u]_{s,p} \coloneqq [u]_{s,p;\R^n}.
\]
With the above norm, $W^{s,p}(\Omega)$ is a reflexive Banach space, see  \cite{Adams,DD}.

The previous fractional space is a good candidate to find ``weak solutions'' to problem \eqref{eq:autovalores}. However, to deal with the boundary condition, we preliminarily restrict ourselves to two special subspaces:

\begin{enumerate}[(i)]
\item $W_0^{s,p}(\Omega)$: the closure in $W^{s,p}(\Omega)$ of the space $C^\infty_c(\Omega)$;
\item $\widetilde{W}^{s,p}(\Omega)$: the space of all $u \in W^{s,p}(\Omega)$ such that $\tilde u \in W^{s,p}(\R^n)$, where $\tilde u$ is the extension	by zero of $u$, outside of $\Omega$. This space is endowed with the norm
\[
\|u\|_{\widetilde{W}^{s,p}(\Omega)} \coloneqq \|\tilde u\|_{s,p} = \|\tilde u\|_{W^{s,p}(\R^n)}.
\] 
\end{enumerate} 

\begin{rem} \label{zero-extension}
From now on, given $u \in \widetilde{W}^{s,p}(\Omega)$ we implicitly suppose that it is defined in the whole space $\R^n$ extending by zero outside of $\Omega$; moreover, we denote this extension by the same letter $u$.
\end{rem}
	
The next result relates the spaces in (i) and (ii). For the proof we refer the reader to \cite[Corollary 1.4.4.5]{Grisvard}.

\begin{thm}\label{thm:teouce}
Let $\Omega\subset\mathbb{R}^n$ be bounded open set with Lipschitz boundary. If $s\neq\frac{1}{p}$, then
\[
  W^{s,p}_0(\Omega) = \widetilde{W}^{s,p}(\Omega).
\]
Furthermore, when $0< s <\frac{1}{p}$ we have
\[
  W^{s,p}_0(\Omega) = \widetilde{W}^{s,p}(\Omega) = W^{s,p}(\Omega).
\]  
\end{thm}
	
The following results are fractional versions of the classical embedding theorems, they can be found in \cite[Corollary~4.53 and Theorem~ 4.54]{DD}, see also \cite{Adams}. We first need the concept of extension domain.

\begin{defn}[Extension domain]
We say that an open set $\Omega \subset \R^n$ is an extension domain for $W^{s,p}$ if there exists a positive constant $C=C(n,s,p,\Omega)$ such that: for every function $u \in W^{s,p}(\Omega)$ there exists $\tilde u \in W^{s,p}(\R^n)$ with $\tilde u(x) = u(x)$ for all $x \in \Omega$ and $\| \tilde u \|_{s,p} \le C \| u \|_{s,p;\Omega}$. Some important examples of extension domains are the bounded domains with Lipschitz boundary, see \cite[Section 1.2]{Grisvard}.
\end{defn}

Let us recall also the definition of the fractional Sobolev conjugate of $p$:
\[
  p_s^* = \begin{cases}
		    	\dfrac{np}{n-sp} &\text{if } sp<n,\\
				\infty & \text{if } sp\ge n.	
		  \end{cases}
\]

\begin{thm}\label{thm:teoembcont}
Let $\Omega \subset \R^n$ be an extension domain for $W^{s,p}$. Then we have: 
\begin{itemize}
  \item if $sp < n$, $W^{s,p}(\Omega)$ is continuously embedded in $L^q(\Omega)$ for any $q \in [p,p_s^*];$ 
  \item if $sp = n$, $W^{s,p}(\Omega)$ is continuously embedded in $L^q(\Omega)$ for any $q \in [p,\infty);$ 
  \item If $sp>n$, $W^{s,p}(\Omega)$ is continuously embedded in $C^{0,\alpha}(\overline{\Omega})$ for any $\alpha \in (0,s-\tfrac{n}{p}]$.
\end{itemize}
	\end{thm}
	
\begin{thm}\label{thm:teoemb}
Let $\Omega \subset \R^n$ be a bounded extension domain for $W^{s,p}$. Then we have:
\begin{itemize}
  \item if $sp\le n$, the embedding of $W^{s,p}(\Omega)$ into $L^q(\Omega)$ is compact for every $q\in[1,p_s^*)$;
  \item if $sp>n$, the embedding of $W^{s,p}(\Omega)$ into $C^{0,\alpha}(\overline{\Omega})$ is compact for $\alpha\in(0,s-\tfrac{n}p)$.  
\end{itemize}
\end{thm}
	
\begin{rem}\label{rem:hipenq}
let $\Omega\subset\mathbb{R}^n$ be a bounded extension domain. Observe that the embedding $W^{s,p}(\Omega)$ into $L^p(\Omega)$ is compact for all $p\in(1,\infty)$. Additionally, if $\max\{1,\tfrac{n}{sp}\}<q < \infty$ then
\[
  \begin{cases}
    pq^{\prime}<p_s^* &\text{ if } sp \le n, \\
	pq^{\prime}\le\infty &\text{ if } sp > n,
  \end{cases}
\]
where $p'$ is the conjugate exponent of $p$, $\tfrac{1}{p'} + \tfrac{1}{p} = 1$. Thus, by Theorem~\ref{thm:teoemb}, we have that the embedding of $W^{s,p}(\Omega)$ into $L^{pq'}(\Omega)$ is compact.
	\end{rem}

In order to work with weak solutions to \eqref{eq:autovalores} we need to find the weak formulation of the operator $(-\Delta_p)^s$ defined in \eqref{eq:splap}.

So first we need to extend the definition of $(-\Delta_p)^s$ to the space $W^{s,p}(\R^n)$ with values in the dual $(W^{s,p}(\R^n))' = W^{-s,p'}(\R^n)$.

This computation is rather direct and we include the details for the sake of completeness.

We begin with a preliminary lemma.
\begin{lem}\label{epsilon}
Given $\ve>0$ we define the approximating operators $(-\Delta_p)_\ve^s$ as
$$
(-\Delta_p)_\ve^s u(x):= \int_{\R^n\setminus B_\ve(x)} \frac{|u(x)-u(y)|^{p-2}(u(x)-u(y))}{|x-y|^{n+sp}}\, \di y.
$$
Then, this operator is well defined between $W^{s,p}(\R^n)$ and $L^{p'}(\R^n)$. Moreover, the following estimate holds,
$$
\|(-\Delta_p)_\ve^s u\|_{p'}\le C [u]_{s,p}^{\frac{p}{p'}},
$$
where $C>0$ depends on $\ve, p, n$ and $s$.
\end{lem}

\begin{proof}
Take $u \in W^{s,p}(\R^n)$ and $\ve>0$. We have,
\begin{align*}
|(-\Delta_p)^s_\ve u(x)| & \le \int_{|x-y|>\ve} \frac{|u(x)-u(y)|^{p-1}}{|x-y|^{n+sp}}\, \di y\\
&\le \left(\int_{|x-y|>\ve} \frac{|u(x)-u(y)|^{p}}{|x-y|^{n+sp}}\, \di y\right)^{\frac{1}{p'}} \left(\int_{|x-y|>\ve} \frac{1}{|x-y|^{n+sp}}\, \di y\right)^{\frac{1}{p}},
\end{align*}
and so
$$
\|(-\Delta_p)^s_\ve u\|_{p'} \le C_{\ve, p, n, s} [u]_{s,p}^\frac{p}{p'},
$$
where 
$$
C_{\ve, p, n, s} = \left(\frac{n\omega_n}{sp}\right)^\frac{1}{p} \ve^{-s},
$$
as we wanted to show.
\end{proof}

In order to properly define the operator $(-\Delta_p)^s$, we use the canonical inclusion of $L^{p'}(\R^n)$ into $W^{-s,p'}(\R^n)$ given by
$$
\langle f, v\rangle := \int_{\R^n} fv\, dx
$$
for every $v\in W^{s,p}(\R^n)$ and therefore, we may define
$$
(-\Delta_p)^s u = \lim_{\ve\to 0} (-\Delta_p)_\ve^s u
$$
in $W^{-s,p'}(\R^n)$.

\begin{prop}\label{weak.form}
For every $u\in W^{s,p}(\R^n)$, there exists the limit $\lim_{\ve\to 0} (-\Delta_p)_\ve^s u$ in $W^{-s,p'}(\R^n)$.
Moreover, it holds
\begin{align*}
\langle (-\Delta_p)^s u, v\rangle &= \lim_{\ve\to 0}\langle (-\Delta_p)_\ve^s u, v\rangle\\
&= \frac12 \iint_{\R^n\times \R^n}\dfrac{|u(x)-u(y)|^{p-2}(u(x)-u(y))(v(x)-v(y))}{|x-y|^{n+sp}} \, \di x \di y.
\end{align*}
\end{prop}

\begin{proof}
Let $v\in W^{s,p}(\R^n)$. Hence,
\begin{align*}
\langle (-\Delta_p)^s_\ve u, v\rangle &= \int_{\R^n} (-\Delta_p)^s_\ve u(x) v(x)\, \di x\\
&= \int_{\R^n}\int_{|x-y|>\ve} \frac{|u(x)-u(y)|^{p-2}(u(x)-u(y)) v(x)}{|x-y|^{n+sp}}\, \di x\di y.
\end{align*}
Analogously,
$$
\langle (-\Delta_p)^s_\ve u, v\rangle = -\int_{\R^n}\int_{|x-y|>\ve} \frac{|u(x)-u(y)|^{p-2}(u(x)-u(y)) v(y)}{|x-y|^{n+sp}}\, \di x\di y.
$$
Therefore
$$
\langle (-\Delta_p)^s_\ve u, v\rangle = \frac12 \int_{\R^n}\int_{|x-y|>\ve} \dfrac{|u(x)-u(y)|^{p-2}(u(x)-u(y))(v(x)-v(y))}{|x-y|^{n+sp}} \, \di x \di y.
$$
From this last equality, the result follows passing to the limit $\ve\downarrow 0$.
\end{proof}


\subsection{A minimum principle} 
Let $\Omega$ be bounded extension domain for $W^{s,p}$, and $V \in L^q(\Omega)$ with $q \in (1,\infty) \cap (\tfrac{n}{sp},\infty)$. We say that $u \in \widetilde{W}^{s,p}(\Omega)$ is a weak super-solution to
\begin{equation}\label{eq:supersol}
\left\lbrace 
  \begin{aligned}
    (-\Delta_p)^s u + V(x) |u|^{p-2}u &= 0  && \text{ in } \Omega,\\
			                        u &= 0  && \text{ in } \R^n \setminus \Omega,
  \end{aligned}
\right. 
\end{equation}
if
\begin{equation}\label{eq:wsupersol}
  \h (u,v) + \int_{\Omega} V(x) |u|^{p-2}u v \, \di x \ge 0 \quad
  \forall v \in \widetilde{W}^{s,p}(\Omega), v \ge 0,
\end{equation}
where $\h \colon W^{s,p}(\R^n)\times
	W^{s,p}(\R^n) \to \R$ is defined as
\[
  \h (u,v) = \frac{1}{2} \iint_{\R^n\times \R^n} \dfrac{|u(x)-u(y)|^{p-2}(u(x)-u(y))(v(x)-v(y))}{|x-y|^{n+sp}} \, \di x \di y.
\]
Observe that by virtue of Proposition \ref{weak.form} this is equivalent to say that $u\in \widetilde{W}^{s,p}(\Omega)$ is a distributional super-solution to \eqref{eq:supersol}.

Notice that $u,v \in \widetilde{W}^{s,p}(\Omega)$ are defined in the whole space,
since we consider them to be extended by zero outside of $\Omega$, see Remark~\ref{zero-extension}. With this convention in mind, observe that 
\[
  \h (u,u) = \frac{1}{2} [u]_{s,p}^p  \quad \text{for all } u \in \widetilde{W}^{s,p}(\Omega).
\]
	
Let us now prove a minimum principle for weak super-solutions of \eqref{eq:supersol}. To this end, we follow the ideas in \cite{BF} and prove first the next logarithmic lemma (see \cite[Lemma 1.3]{di2014local}). Although this is not the more general version of the logarithmic lemma (c.f. with \cite[Lemma 1.3]{di2014local}) it will suffices our purposes and simplifies the presentation.
	
\begin{lem}\label{lema:DKP}
Let $\Omega$ be bounded extension domain for $W^{s,p}$, and $V \in L^q(\Omega)$ with $q \in (1,\infty) \cap (\tfrac{n}{sp},\infty)$. Suppose that $u$ is a nonnegative weak super-solution of  \eqref{eq:supersol}. Then for any $B_r=B_r(x_0)$ such that $B_{2r}\subset \Omega$ and $0<\delta<1$
\begin{align*}
\iint_{B_r\times B_r}& \dfrac{1}{|x-y|^{n+sp}} \left| \log \left(\dfrac{u(x)+\delta}{u(y)+\delta}\right)\right|^p \, \di x \di y \le Cr^{n-sp} + \|V\|_{1;B_{2r}},
\end{align*}
where $C$ depends only on $n,s,$ and $p$.
\end{lem}
	
\begin{proof}
Let $\delta>0$ and $\phi\in C_0^\infty( B_{\tfrac{3r}2})$ be such that 	
\[
  0\le \phi \le 1, \quad \phi\equiv1 \text{ in } B_r \quad \text{ and } \quad|D\phi|<Cr^{-1} \text{ in } B_{\tfrac{3r}2}\subset B_{2r}. 
\]
Taking $v=(u+\delta)^{1-p}\phi^p$ as test function in \eqref{eq:wsupersol} we have that
\begin{equation}\label{eq:DKP}
  \begin{aligned}
    -\int_{B_{\tfrac{3r}2}} V(x) \frac{u^{p-1}}{(u+\delta)^{p-1}} \phi^p\, \di x \le \h (u,(u+\delta)^{1-p}\phi^p). 
  \end{aligned}
\end{equation}
In the proof of Lemma 1.3 in \cite{di2014local}, it is showed that
\begin{align*}
  \h (u,(u+\delta)^{1-p}\phi^p) \le &\ Cr^{n-sp} - \iint_{B_r\times B_r} \dfrac{1}{|x-y|^{n+sp}} \left| \log\left(\dfrac{u(x)+\delta}{u(y)+\delta}\right)\right|^p \, \di x \di y,
\end{align*}
where $C$ depends only on $n,s,$ and $p$. 

Then, by \eqref{eq:DKP} and using that $0\le u^{p-1}(u+\delta)^{1-p}\phi^p\le1$ in $B_{\tfrac{3r}{2}},$ the lemma holds.
\end{proof}
	
Proceeding as in the proof of Theorem~A.1 in \cite{BF} and using the previous lemma, we get the following minimum principle.

\begin{thm}\label{thm:mprinciple} 
Under the hypotheses of the previous lemma, if $u$ is a nonnegative weak super-solution of \eqref{eq:supersol} and $u\not\equiv0$ in $\Omega$, then $u>0$ a.e in $\Omega$. 
\end{thm}

\begin{proof}
Assume first that $u\not\equiv0$ in all connected components of $\Omega$.

		We argue by contradiction and we assume that 
		$Z=\{x\in \Omega\colon u(x)=0\}$ has positive measure. Since 
		$u\not\equiv0$ in all connected 
		components of $\Omega,$ there are a ball
		$B_R=B_R(x_0)\subset\Omega$ and 
		$r\in(0,\nicefrac{R}2)$ such that $|B_r\cap Z|>0$ 
		and $u\not\equiv0$ in $B_r.$ 
		
		For any $\delta>0$ and $x\in\R^n,$ we define 
		\[
			F_{\delta}(x)\coloneqq\log\left(1+\dfrac{u(x)}{\delta}
			\right).
		\] 
		Observe that, if $y\in B_r\cap Z $ then
		\[
			|F_\delta(x)|^p=|F_\delta(x)-F_\delta(y)|^p
			\le \dfrac{(2r)^{n+sp}}{|x-y|^{n+sp}} \left|
			\log\left(\dfrac{u(x)+\delta}{u(y)+\delta}\right)\right|^p
			\quad\forall x\in B_r.
		\]
		Then
		\[
			|F_{\delta}(x)|^p\le\dfrac{(2r)^{n+sp}}{|Z\cap B_r|}
			\int_{B_r}\dfrac{1}{|x-y|^{n+sp}} \left|
			\log\left(\dfrac{u(x)+\delta}{u(y)+\delta}\right)\right|^p
			\di y\quad\forall x\in B_r.
		\]
		Therefore
		\[
			\int_{B_r}
			|F_{\delta}(x)|^p dx\le\dfrac{(2r)^{n+sp}}{|Z\cap B_r|}
			\iint_{B_r\times B_r}\dfrac{1}{|x-y|^{n+sp}} \left|
			\log\left(\dfrac{u(x)+\delta}{u(y)+\delta}\right)\right|^p
			\di x\di y
		\]
		By, Lemma \ref{lema:DKP}, there is a constant $C$ independent
		of $\delta$ such that 
		\[
			\int_{B_r}
				|F_{\delta}(x)|^p dx\le
				C\dfrac{r^{n}(r^n+r^{sp}\|V\|_{1;B_{2r}})}{|Z\cap B_r|}
		\]
		Taking $\delta \to 0$  in the above inequality, we obtain
		\[
			u\equiv0 \mbox{ in } B_r
		\]
		which is a contradiction since $u\not\equiv0$ in $B_r.$
		Thus $u>0$ in $\Omega.$
		
Now, suppose that there is $Z$ a connected components of $\Omega$ such that $u\equiv0$ in $Z.$ Taking $\phi\in C_c^{\infty}(Z)$, $\phi>0$, as a test function we get
	\[
		\h(u,\phi)\ge 0.
	\]
Now, observe that in this case,
$$
\h(u,\phi) = -2\int_{\Omega\setminus Z} (u(x))^{p-1} \left(\int_Z \frac{\phi(y)}{|x-y|^{n+sp}}\, \di y\right)\, \di x.
$$
	Therefore
	\[
		\int_{\Omega\setminus Z}(u(x))^{p-1}
		\int_{Z}\dfrac{\phi(y)}{|x-y|^{n+sp}}\di y\di x\le 0
		\quad\forall \phi\in C_c^{\infty}(Z),\ \phi>0.
	\]
	Then $u=0$ in $\Omega,$ which is a contradiction.
\end{proof}


\section{The first eigenvalue}\label{elprimeraut}

Throughout this section, $\Omega \subset \R^n$ shall be a 
bounded extension domain boundary and $V \in L^q(\Omega)$, $q \in (1,\infty) \cap 
(\tfrac{n}{sp},\infty)$. We say that a function $u\in \widetilde{W}^{s,p}
(\Omega)$ is a weak solution of \eqref{eq:autovalores} if 
\begin{equation}\label{eq:plantdebil}
	\h (u,v) + \int_{\Omega} V(x)|u|^{p-2}uv\, \di x = \lambda 
	\int_{\R^n}|u|^{p-2}u v\, \di x	
\end{equation}
for all $v\in \widetilde{W}^{s,p}(\Omega)$. In this context, we say that $
\lambda\in\mathbb{R}$ is an {\em eigenvalue} provided there exists a 
nontrivial weak solution $u\in\widetilde{W}^{s,p}(\Omega)$ of \eqref
{eq:autovalores}. The function $u$ is a corresponding eigenfunction.

For a study of this first eigenvalue and its associated eigenfunction in the case $V=0$ we refer to \cite{BLP}.

Now, by Theorem~\ref{thm:mprinciple}, we have 
that
 
\begin{lem}\label{lema:signo-constante}
	If $u$ is a nonnegative 
		eigenfunction associated to 
		$\lambda$ then $u>0$ a.e. in $\Omega.$
\end{lem}

 Now, our goal is to prove that the lowest (first) eigenvalue of \eqref{eq:autovalores} is
\begin{equation} \label{first-eigenvalue}
	\lambda(V)=\inf \left\{\frac{1}{2} [u]_{s,p}^p + 
	\int_{\Omega} V(x)|u|^p\, \di x \colon u 
	\in\widetilde{W}^{s,p}(\Omega) \text{ and } \|u\|_p = 1\right\}.
\end{equation}
The next lemma implies that $\lambda(V)$ is well defined.
	
\begin{lem}\label{lema:aux1}
  Let $\Omega \subset \R^n$ be a bounded extension domain. Then, given $\varepsilon>0$, there is a constant $C_\varepsilon > 0$ such that
  \[
    \left| \int_{\Omega}V(x)|u|^p\, \di x \right| \le 
	\varepsilon [u]_{s,p}^p + C_\varepsilon \|V\|_{q;\Omega}
	\|u\|_p^p
  \]
  for all $u \in \widetilde{W}^{s,p}(\Omega)$.
\end{lem}
\begin{proof}
  The lemma is trivial for $V\equiv 0$, so let us suppose that $V \not\equiv 0$. Assume by contradiction that there exist $\varepsilon_0>0$ and a sequence $\{u_k\}_{k \in \N} \subset \widetilde{W}^{s,p}(\Omega)$ such $\|u_k\|_{pq'} = 1$ and
  \[
    \varepsilon_0 [u_k]_{s,p}^p + k\|V\|_{q;\Omega} \|u_k\|_p^p
	\le \left|\int_{\Omega}V(x)|u_k|^p\, dx\right| \quad \text{for all } k \in \N.
  \]
  Then, by H\"older inequality,
  \begin{equation}\label{eq:la1.1}
    \varepsilon_0 [u_k]_{s,p}^p + k\|V\|_{q;\Omega} \|u_k\|_p^p
	\le \|V\|_{q;\Omega} \|u_k\|_{pq'}^p \quad \text{for all } k \in \N.
  \end{equation}
  Therefore $\{u_k\}_{k \in \N}$ is bounded in $\widetilde W^{s,p}(\Omega)$ and 
  \begin{equation}\label{eq:lau1.2}
    u_k\to 0 \text{ in } L^p(\R^n).
  \end{equation}
  
  Now, as $\widetilde W^{s,p}(\Omega)$ is continuously embedded in $W^{s,p}(\Omega)$, and this compactly in $L^{pq'}(\Omega)$ (by Theorem~\ref{thm:teoemb} and Remark~\ref{rem:hipenq}), there exist a subsequence (still denoted by $\{u_k\}_{k\in\mathbb{N}})$, and some $u\in L^{pq^\prime}(\Omega)$ such that $u_k\to u$ in $L^{pq^\prime}(\Omega)$. Then  $\|u\|_{pq'} = 1$, which contradicts \eqref{eq:lau1.2} and completes the proofs.
  \end{proof}
	
Using the previous lemma and standard compactness argument, see \cite[Theorem~2.7]{FBDP}, it follows that there is an eigenfunction associated to $\lambda(V)$, as the next theorem states.

\begin{thm}\label{thm:aut.existencia}
	Let $\Omega \subset \R^n$ be a bounded extension domain. Then there exists $u \in \widetilde{W}^{s,p}(\Omega)$ such that $\|u\|_p = 1$ and
	\[
		\lambda(V) = \dfrac12[u]_{s,p}^p + \int_{\Omega} V(x)|u|^p\, \di x.
	\]
	Moreover, $u$ is an eigenfunction associated to 
	$\lambda(V)$.
\end{thm}

\begin{rem} \label{constant-sign}
  Any eigenfunction $u$ constructed in the previous theorem can be chosen to be positive. Indeed, as $||u(x)|-|u(y)|| \le |u(x) - u(y)|$ for all $x,y \in \R^n$, then
  \[
    \dfrac12[|u|]_{s,p}^p + \int_{\Omega} V(x)|u|^p\, \di x \le 
    \dfrac12[u]_{s,p}^p + \int_{\Omega} V(x)|u|^p\, \di x = \lambda(V).
  \]  
  This implies that $|u|$ is an eigenfunction associated to $\lambda(V)$. And by Lemma \ref{lema:signo-constante},  $|u| > 0$. Actually, Theorem~\ref{thm:autoval1} below shows that all the eigenfunctions associated to $\lambda(V)$ have constant sign. 
\end{rem}

A key ingredient in the next sections is the simplicity of the first eigenvalue $\lambda(V)$. In order to prove this result we need the following Picone-type identity (see Lemma~6.2  in \cite{Amghibech}).

\begin{lem}\label{lem:laux1}  
Let $p\in(1,\infty)$. For $u,v\colon\Omega\to\mathbb{R}$ such that $u \ge 0$ and $v>0$,  we have
\[
	L(u,v)\ge0 \quad \mbox{in } \Omega\times \Omega,
\]
where
\[
	L(u,v)(x,y) = |u(x)-u(y)|^p-|v(x)-v(y)|^{p-2}(v(x) - v(y)) \left(\dfrac{u^p(x)}{v^{p-1}(x)} -\dfrac{u^p(y)}{v^{p-1}(y)}\right). 
\]
The equality holds if and only if $u=kv$  in $\Omega$ 
for some constant $k$.
\end{lem}
	
\begin{thm}\label{thm:autoval1}
	Let $\Omega \subset \R^n$ be a bounded extension domain. 
	Assume that $u$ is a positive eigenfunction 
	corresponding to $\lambda(V)$ (see Remark~\ref{constant-sign}). Then if $
	\lambda > 0$ is such that there is a nonnegative eigenfunction $v$ of 
	\eqref{eq:autovalores} with eigenvalue $\lambda,$ then 
	$\lambda = \lambda(V)$ and there is  $k \in \R$ 
	such that $v = ku$ a.e. in $\Omega$. 
\end{thm}
\begin{proof} 
  From the definition of $\lambda(V)$, immediately follows that
  $\lambda(V) \le \lambda$. On the other hand, by Lemma  
  \ref{lema:signo-constante}, $v>0$ in $\Omega.$
  
  Let $m \in \N$ and $v_m\coloneqq v + \tfrac1{m}$. We begin by proving that 
  $w_{m} \coloneqq u^{p} / v_m^{p-1}\in\widetilde{W}^{s,p}(\Omega)$.
 First observe that that  $w_m=0$ in $\R^n \setminus \Omega$ and $w_{m}\in L^{p}(\Omega),$ due to $u\in L^{\infty}(\Omega)$, see Lemma \ref{lema:cotainfty}. Now, for all $(x,y) \in \R^n \times \R^n$ we have
  \begin{align*}
    |w_{m}(x)-w_{m}(y)| =& \left| \dfrac{u^p(x) - u^p(y)}{v_m^{p-1}(x)} - \dfrac{u^p(y) \left( v_m ^{p-1}(x) - v_m^{p-1}(y) \right)}{v_m^{p-1}(x) v_m^{p-1}(y)} \right| \\
		\le &\ m^{p-1} \left|u^p(x) - u^p(y) \right| + \|u\|_{\infty}^p \dfrac{\left|v_m^{p-1}(x) - v_m^{p-1}(y) \right|}{v_m^{p-1}(x)v_m^{p-1}(y)} \\
		\le&\ p m^{p-1} (u^{p-1}(x) + u^{p-1}(y))|u(x) - u(y)| \\
		& + (p-1)\|u\|_{\infty}^p \dfrac{v_m^{p-2}(x) + v_m^{p-2}(y)}{v_m^{p-1}(x) v_m^{p-1}(y)}|v_m(x)-v_m(y)|\\
		\le &\ 2pm^{p-1} \|u\|_{\infty}^{p-1} |u(x)-u(y)| \\
		& + (p-1) \|u\|_{\infty}^p \left(\dfrac1{v_m(x) v_m^{p-1}(y)} + \dfrac1{v_m^{p-1}(x) v_m(y)} \right)|v(x)-v(y)| \\
	 	\le &\ C(m,p,\|u\|_{\infty})
		\left( |u(x)-u(y)|+|v(x)-v(y)| \right)
	  \end{align*} 
  As $u,v \in \widetilde W^{s,p}(\Omega)$, we deduce that $w_m \in \widetilde W^{s,p}(\Omega)$ for all $m \in \N$.

  Recall that $u,v\in\widetilde{W}^{s,p}(\Omega)$ are two eigenfunctions of problem  \eqref{eq:autovalores}  with eigenvalue $\lambda(V)$ and 	$\lambda$ respectively. Then, by using the previous lemma, we deduce that 
  \begin{align*}
    0 \le& \dfrac12 \iint_{\Omega\times \Omega} \dfrac{L(u,v_m)(x,y)}{|x-y|^{n+sp}} \, \di x \di y \\
    	\le &\ \frac{1}{2} \iint_{\R^n\times \R^n} \dfrac{|u(x)-u(y)|^p}{|x-y|^{n+sp}} \, \di x \di y \\ 
    	& - \frac{1}{2} \iint_{\R^n\times \R^n} \dfrac{|v(x)-v(y)|^{p-2}(v(x)-v(y))}{|x-y|^{n+sp}}	\left(\dfrac{u^p(x)}{v_m^{p-1}(x)} - \dfrac{u^p(y)}{v_m^{p-1}(y)} \right) \, \di x \di y \\
		\le &\ \lambda(V) \int_{\Omega} u^p \, \di x - \int_{\Omega}V(x)u^p \, \di x - \lambda \int_{\Omega} v^{p-1} \dfrac{u^p}{v_m^{p-1}}\, \di x + \int_{\Omega}V(x)v^{p-1} \dfrac{u^p}{v_m^{p-1}}\, \di x.
	\end{align*}
	Taking $m \to \infty$ and using Fatou's lemma and the dominated convergence theorem, we infer that
	\[
      \iint_{\R^n\times \R^n} \dfrac{L(u,v)(x,y)}{|x-y|^{n+sp}} \, \di x \di y = 0
	\]
	(recall that $\lambda(V)\le \lambda$).
	
	Therefore, by the previous lemma, $L(u,v)(x,y)=0$ a.e. and $u=kv$ for 
	some constant $k>0$.
\end{proof}
	
\begin{rem}\label{remark:unicidad}
  As a consequence of the previous theorem, $\lambda(V)$ is simple and there 
  is a unique 
  associated positive eigenfunction $u \in \widetilde{W}^{s,p}(\Omega)$ such that $\|u\|_p = 1$.
\end{rem}

To conclude this section, we prove that $\lambda(V)$ is isolated. To this end, we follow the ideas in \cite{LL} and first provide a lower bound for the measure of the nodal sets.
	
\begin{lem}\label{lema:nodal}
  Let $\Omega \subset \R^n$ be be a  bounded extension domain. 
  If $u$ is an eigenfunction associated to $\lambda > \lambda(V)$, then 
  \[
    \min\left\{A(\lambda)^{\tfrac{1}{(1-\tfrac{p}{r})}}, A(\lambda)^{\tfrac{1}{(\tfrac{1}{q'} - \tfrac{p}{r})}} \right\} \le |\Omega_{\pm}|,
  \]
  where $r \in (pq^\prime,p_s^*)$, $A(\lambda) \coloneqq (C(|\lambda|+1+\|V\|_{q;\Omega}))^{-1}$, $C$ is a constant independent of $V$, $\lambda$ and $u$, and $|\Omega_{\pm}|$ is the Lebesgue measure of $\Omega_{\pm}=\{x\in\R^n \colon u_{\pm}(x)\neq0\}$.	
\end{lem}
\begin{proof}
  According to Theorem~\ref{thm:autoval1}, $u_+$ and $u_-$ are not trivial. We shall prove the inequality for $|\Omega_+|$, the proof of the other inequality is similar.
  
  Observe that $u_+\in\widetilde{W}^{s,p}(\Omega)$ and
  \[
    |u_+(x)-u_+(y)|^p \le |u(x)-u(y)|^{p-2}(u(x)-u(y))(u_+(x)-u_+(y))
  \]
  for all $(x,y) \in \R^n \times \R^n$. Let us recall Remark~\ref{zero-extension} to keep in mind that $u=0$ in $\R^n\setminus\Omega$. Then, using H\"older's inequality, we have
  \begin{equation}\label{eq:lb1}
    \begin{aligned}
	  \frac{1}{2}[u_+]_{s,p}^p& \le \h(u,u_+)\\
	  &= \lambda \int_{\Omega}u_+^p\, \di x - \int_{\Omega} V(x)u_+^p \di x \\
	  &\le \lambda \int_{\Omega}u_+^p\, \di x + \int_{\Omega} V_-(x)u_+^p \di x \\
	  &\le |\lambda| \|u_+\|_p^p + \|V\|_{q;\Omega} \|u_+\|_{pq'}^p \\
	  &\le \left( |\lambda| |\Omega_+|^{1-\tfrac{p}r} + \|V\|_{q;\Omega} |\Omega_+|^{\tfrac1{q^\prime}-\tfrac{p}{r}} \right) \|u_+\|_r^p.
			\end{aligned}
		\end{equation} 
		
  On the other hand, by Theorem~\ref{thm:teoembcont}, there is a constant $C$ independent on $u$ such that
  \[
    \|u_+\|_r \le C[u_+]_{s,p}.
  \]
  This and \eqref{eq:lb1} implies that
  \[
    \|u_+\|_r^p \le 2C \left( (|\lambda|+1)
				|\Omega_+|^{1-\tfrac{p}r}
				+\|V\|_{q;\Omega} 
				|\Omega_+|^{\tfrac1{q^\prime}-\tfrac{p}{r}}
				\right)\|u_+\|_r^p,
  \]
  that is
  \[
    1 \le 2C \left( (|\lambda|+1) |\Omega_+|^{1-\tfrac{p}r} + \|V\|_{L^q(\Omega)} |\Omega_+|^{\tfrac1{q^\prime}-\tfrac{p}{r}} \right).
  \]
  Therefore
		\[
		 	\min\left\{A(\lambda)^{\tfrac{1}{(1-\tfrac{p}{r})}},
		 	 A(\lambda)^{\tfrac{1}{(\tfrac{1}{q^\prime}
		 	 -\tfrac{p}{r})}}\right\}
		 	 \le |\Omega_+|.
		\]
\end{proof}
	
\begin{thm}\label{teo:isolated}
  Let $\Omega \subset \R^n$ be a connected bounded extension domain. 
  Then the first eigenvalue $\lambda(V)$ is isolated.  	
\end{thm}

\begin{proof}
  By definition $\lambda(V)$ is left-isolated. To prove that $\lambda(V)$ is right-isolated, we argue by contradiction. We assume that there exists a a sequence of 
  eigenvalues $\{\lambda_k\}_{k \in \N}$ such that $\lambda_k \searrow \lambda(V)$ as $k\to \infty$. Let $u_k$ be an eigenfunction associated to $\lambda_k$ with $\|u_k\|_p = 1$. Then, thanks to Lemma~\ref{lema:aux1}, $\{u_k\}_{k \in \N}$ is bounded in $\widetilde W^{s,p}(\Omega)$ and therefore we can extract a subsequence (that we still denoted by $\{u_k\}_{k \in \N}$) such that 
  \begin{gather*}
    u_k\rightharpoonup u \text{ weakly in } \widetilde{W}^{s,p}(\Omega), \\
	u_k\to u \text{ in } L^{pq'}(\R^n), \\
	u_k\to u \text{ in } L^p(\R^n).
  \end{gather*}
  Observe that $u_k^p\to u^p $ in $L^{q^\prime}(\R^n)$ since $u_k\to u $ in $L^{pq^\prime}(\R^n)$. Then $\|u\|_p = 1$, and
  \begin{align*}
    \dfrac12[u]_{s,p}^p & \le \liminf_{k\to \infty}\dfrac12[u_k]_{s,p}^p \\
					&=\lim_{k\to \infty}\lambda_k \int_{\R^n} |u_k(x)|^p\, \di x - \int_{\Omega}V(x)|u_k|^p\, \di x\\
					&= \lambda(V) \int_{\R^n} |u(x)|^p\, \di x - \int_{\Omega}V(x)|u|^p\, \di x.
	\end{align*}
  Hence, $u$ is an eigenfunction associated to $\lambda(V)$. By Theorem~\ref{thm:autoval1}, we can assume that $u>0$.
	
  On the other hand, by the Egorov's theorem, for any $\varepsilon>0$ there exists a subset $U_\varepsilon$ of $\Omega$ such that $|U_\varepsilon|<\varepsilon$ and $u_k\to u>0$ uniformly in 
  $\Omega\setminus U_{\varepsilon}.$ This contradicts the previous lemma. Indeed,
  \[
    0 < \lim_{k\to\infty} \min\left\{A(\lambda_k)^{\tfrac{1}{(1-\tfrac{p}{r})}}, A(\lambda_k)^{\tfrac{1}{(\tfrac{1}{q^\prime}-\tfrac{p}{r})}}\right\} 		 	 \le\lim_{k\to\infty}|\{x\in\mathbb{R}^n\colon u_{k}<0\}|,
  \]
where $r\in(pq^\prime,p_s^*)$.
\end{proof}


\section{The  functional $\lambda(V)$}\label{functional}

In this section we shall provide some useful properties of the functional 
$$
\lambda\colon L^q(\Omega)\to \mathbb{R},\quad \max\{1,\tfrac{n}{sp}\} < q < \infty.
$$ 
that associate to every $V \in L^q(\Omega)$ the number $\lambda(V)$ given by \eqref{first-eigenvalue}.
	
From now on, $\Omega \subset \R^n$ denotes  
a bounded extension domain and $V$ is a function in $L^q(\Omega)$, with $\max\{1,\tfrac{n}{sp}\} < q < \infty$.	 
	
\begin{lem}\label{lem:propfunct}
  The functional $\lambda$ is concave  in $L^q(\Omega)$. Moreover, for any $M>0$ there exists a constant $C=C(s,p,q,M)$ such that
  \[
    \lambda(V) \le C
  \]
  for all $V\in L^q(\Omega)$ such that $\|V\|_{q;\Omega} \le M$. 
\end{lem}

\begin{proof}
  Given $V, W\in L^q(\Omega)$, we have by definition that
  \begin{gather*}
			\lambda(V) \le \dfrac12[u]_{s,p}^p + \int_{\Omega} V(x)|u|^p\, \di x, \\
			\lambda(W) \le \dfrac12[u]_{s,p}^p + 
			\int_{\Omega} W(x)|u|^p\, \di x,
  \end{gather*}
  for all $u \in \widetilde{W}^{s,p}(\Omega)$ with $\|u\|_p=1$. Then, for any $t\in(0,1)$
		and $V, W\in L^q(\Omega)$,
  \[
    t\lambda(V) + (1-t)\lambda(W) \le \dfrac12[u]_{s,p}^p + \int_{\Omega} (tV(x)+(1-t)W(x))|u|^p\, \di x \\
  \]
  for all  $u\in \widetilde{W}^{s,p}(\Omega)$ such that $\|u\|_p = 1$. After recalling the definition of the functional $\lambda$, we deduce then that
  \[
    t\lambda(V)+(1-t)\lambda(W) \le	\lambda(tV+(1-t)W),
  \]
  that is, $\lambda$ is concave.
		
  Let us now prove that  $\lambda$ is locally bounded in $L^q(\Omega)$.	Indeed, given $M>0$ and $V \in L^q(\Omega)$ with $\|V\|_{q;\Omega} \le M$, fix a function $\phi \in C_{c}^\infty(\Omega) \subset \widetilde W^{s,p}(\Omega)$ such that $\|\phi\|_p = 1$. Thus,
  \begin{align*}
    \lambda(V) &\le \dfrac12[\phi]_{s,p}^p + \int_{\Omega} V(x)|\phi|^p\, \di x \\
			   &\le \dfrac12[\phi]_{s,p}^p + \|V\|_{q;\Omega} \|\phi\|_{pq'}^p \\
			   &\le \dfrac12[\phi]_{s,p}^p + M \|\phi\|_{pq'}^p.
  \end{align*}
  \end{proof}
	
Our next aim is to show that $\lambda$ is continuous. We'll need the following estimate, related to that in Lemma~\ref{lema:aux1}. The only difference with Lemma \ref{lema:aux1} is the fact that here we need the constants to be uniform with respect to the potential function $V$.
  \begin{lem}\label{lem:caux}
    Given $M>0$, for any $\varepsilon > 0$ there is a constant $C_\varepsilon > 0$ such that
	\begin{equation}\label{eq:caux0}
	  \left|\int_{\Omega} V(x)|u|^p\, \di x \right| \le \varepsilon [u]_{s,p}^p + C_\varepsilon \|V\|_{q;\Omega}
			\|u\|_p^p
    \end{equation}
		for all $u\in \widetilde{W}^{s,p}(\Omega)$ and 
		$V\in L^q(\Omega)$ such
		that $\|V\|_{L^q(\Omega)}\le M.$
  \end{lem}
  \begin{proof}
  Suppose by contradiction that for all $k \in \N$ there exist $\varepsilon_0 > 0$ and a sequence $\{(u_k,V_k)\}_{k \in \N} \subset \widetilde W^{s,p}(\Omega) \times L^q(\Omega)$ such that $\|u_k\|_{pq'} = 1$, $\|V_k\|_{q;\Omega} \le M$ and
  \begin{equation} \label{eq:caux}
    \left| \int_{\Omega}V_k(x)|u_k|^p\, \di x \right| \ge \varepsilon_0 [u_k]_{s,p}^p + k \|V_k\|_{q;\Omega} \|u_k\|_p^p \quad \text{for all } k \in \N.
  \end{equation}
  Then, by H\"older's inequality, we have that
  \begin{gather}
    \varepsilon_0 [u_k]_{s,p}^p + k \|V_k\|_{q;\Omega} \|u_k\|_p^p \le \|V_k\|_{q;\Omega} 			\|u_k\|_{pq'}^p \le M, \label{eq:caux1} \\
    \|u_k\|_{p;\Omega}\le\|u_k\|_{pq^\prime;\Omega} |\Omega|^{\frac{1}{pq}},
  \end{gather}
  for all $k \in \N.$ Therefore $\{(u_k,V_k)\}_{k\in\mathbb{N}}$ is bounded in $\widetilde{W}^{s,p}(\Omega) \times L^q(\Omega)$ and 
  \begin{equation}\label{eq:caux3}
    \lim_{k \to \infty} \|V_k\|_{q;\Omega} \|u_k\|_p^p = 0.
  \end{equation}
Thus, there exist a subsequence (still denoted by  $\{(u_k,V_k)\}_{k \in \N}$) and some $(u,V) \in \widetilde W^{s,p}(\Omega)\times L^q(\Omega)$, such that
  \begin{equation}\label{eq:caux4}
    \begin{gathered}
      V_k\rightharpoonup V \text{ weakly in } L^q(\Omega),\\
	  u_k\rightharpoonup u \text{ weakly in } \widetilde W^{s,p}(\Omega),\\
	  u_k\to u \text{ in } L^{pq^\prime}(\R^n).
    \end{gathered}
  \end{equation}
  This implies that $\|u\|_{pq'} = 1$, $\|V\|_{q;\Omega} \le M$ and
  \begin{gather*}
    |u_k|^p\to |u|^p  \text{ in } L^{q^\prime}(\R^n),\\
	u_k\to u  \text{ in } L^{p}(\R^n).
  \end{gather*}
  Using \eqref{eq:caux3}, we deduce that $\|V\|_{q;\Omega} \|u\|_p = 0$. As $\|u\|_{pq'} = 1$, then $V \equiv 0$.
  
  Therefore $V_k\to 0$ in  $L^q(\Omega)$. Using this and \eqref{eq:caux4} in \eqref{eq:caux}, we deduce that
  \[
    [u]_{s,p}^p \le \liminf_{k\to\infty}	[u_k]_{s,p}^p \le 0,
  \]
  which implies that $u\equiv0$. This contradiction completes the proof. 
  \end{proof}
	
  \begin{lem}\label{lema:cont}
  The functional $\lambda$ is continuous.
  \end{lem}
  \begin{proof}
A proof of this result follows directly from the fact that any convex and locally bounded function in a Banach space is locally H\"older continuous (see \cite{R}). Nevertheless, we include here a direct proof of this fact since some of the arguments will be needed in the sequel.

  Let $V\in L^{q}(\Omega)$ and $\{V_k\}_{k \in \N}$ be a sequence in $L^q(\Omega)$ such that 
  \begin{equation}\label{eq:cont0}
	 V_k \to V \text{ in }  L^q(\Omega). 
  \end{equation} 
  Let us prove that $\lambda(V_k) \to \lambda(V)$ as $k \to \infty$.
		 
  Let $\{u_k\}_{k \in \N}$ be a sequence in $\widetilde{W}^{s,p}(\Omega)$ such that $\|u_k\|_p = 1$ and
  \[
    \lambda(V_k) = \dfrac12[u_k]_{s,p}^p + \int_{\Omega}V_k(x)|u_k|^p\, \di x \quad \text{for all } k \in \N.
  \]
  Then, for any $k \in \N$ and  $u\in\widetilde{W}^{s,p}(\Omega)$ such that $\|u\|_{L^p(\Omega)}=1$, 
  \[
    \lambda(V_k) \le \dfrac12[u]_{s,p}^p + \int_{\Omega} V_k(x)|u|^p\, \di x.
  \]
  Thus, using \eqref{eq:cont0}, we deduce that
  \[
    \limsup_{k\to\infty}\lambda(V_k) \le \dfrac12[u]_{s,p}^p + \int_{\Omega} V(x)|u|^p\, \di x
  \]
  for all $u \in \widetilde W^{s,p}(\Omega)$ with $\|u\|_p = 1$. Hence  
  \begin{equation}\label{eq:cont1}
    \limsup_{k\to\infty}\lambda(V_k)\le \lambda(V).
  \end{equation}
		
  Now, let us take a subsequence $\{V_{k_{j}}\}_{j \in \N}$ of $\{V_{k}\}_{k\in\mathbb{N}}$ so that
  \begin{equation}\label{eq:cont2}
    \lim_{j\to\infty}\lambda(V_{k_j}) = \liminf_{k \to \infty} \lambda(V_k).
  \end{equation}
  By \eqref{eq:cont0}, we can assume that for any $j \in \N$ we have that $\|V_{k_j}\|_{L^q(\Omega)}\le M$ for some suitable constant $M$. Then, by Lemmas \ref{lem:propfunct} and \ref{lem:caux}, there exist positive constants $C$ and $D$ independent of $j$ such that
  \begin{align*}
    C &\ge \lambda(V_{k_j})\\
	  &= \dfrac12
	  [u_{k_j}]_{s,p}^p + \int_{\Omega} V_{k_j}(x)|u_{k_j}|^p\, \di x \\
	  &\ge \dfrac12
	  [u_{k_j}]_{s,p}^p - \dfrac14[u_{k_j}]_{s,p}^p - D \|V_{k_j}\|_{q;\Omega} \|u_{k_j}\|_p^p.
  \end{align*}
  Therefore
  \[
    [u_{k_j}]_{s,p}^p \le 4(C+DM)
  \]
  for all $j \in \N$. Then, $\{u_{k_j}\}_{j \in \N}$ is bounded in $\widetilde{W}^{s,p}(\Omega)$ and there exist a subsequence (still denoted by $\{u_{k_j}\}_{j \in \N}$) and some $u \in \widetilde W^{s,p}(\Omega)$ such that
  \begin{gather*}
    u_{k_j}\rightharpoonup u \text{ weakly in } \widetilde{W}^{s,p}(\Omega),\\
    u_{k_j}\to u  \text{ strongly in } L^{p}(\R^n),\\
	u_{k_j}\to u \text{ strongly in } L^{pq'}(\R^n).
  \end{gather*}
  Thus $\|u\|_p=1$ and 
  \begin{gather*}
    u_{k_j}^p\to u^p  \text{ strongly  in } L^{q^\prime}(\R^n),\\
  \end{gather*}
  Now, using \eqref{eq:cont0} and \eqref{eq:cont2}, 
  we have that
  \begin{align*}
    \liminf_{k\to \infty}\lambda(V_k)
			&= \lim_{j\to \infty}\lambda(V_{k_j})
			= \lim_{j\to \infty} \dfrac12[u_{k_j}]_{s,p}^p + 
			\int_{\Omega} V_{k_j}(x)|u_{k_j}|^p\, \di x \\
		 	&\ge \dfrac12[u]_{s,p}^p + \int_{\Omega} V(x)|u|^p\, \di x
		 	\ge \lambda(V).
  \end{align*}
  This and \eqref{eq:cont1}, imply that
  \[
    \lim_{k\to \infty}\lambda(V_k) = \lambda(V);
  \]
  and the proof is complete.
  \end{proof}
	
  \begin{rem}\label{rem:conv}
    Let $V\in L^q(\Omega)$, and $\{V_k\}_{k \in \N}$ be a sequence in $L^q(\Omega)$ such that $V_k \to V$. Suppose that $\{u_k\}_{k \in \N} \subset \widetilde W^{s,p}(\Omega)$, is the sequence of the positive eigenfunctions associated to $\lambda(V_k)$ with $\|u_k\|_p = 1$. Then, proceeding as in the proof of the previous lemma, it is possible to extract a subsequence $\{u_{k_j}\}_{j \in \N}$ such that 
    \begin{gather*}
	  u_{k_j} \rightharpoonup u  \text{ weakly in } 
	  \widetilde W^{s,p}(\Omega), \\
	  u_{k_j} \to u  \text{ strongly in } L^{p}(\R^n),\\
	  u_{k_j} \to u  \text{ strongly in } L^{pq'}(\R^n). \\
	\end{gather*}
    Therefore
	\begin{align*}
		\lambda(V) = \lim_{j\to\infty}\lambda(V_{k_j}) &= 
		\lim_{j\to\infty} \dfrac12[u_{k_j}]_{s,p}^p + 
		\int_{\Omega} V_{k_j}(x)|u_{k_j}|^p \, \di x \\
		&\ge \dfrac12[u]_{s,p}^p + \int_{\Omega} V(x)|u|^p \, \di x \\
		&\ge \lambda(V).
    \end{align*}
	
	Then $u$ is the positive eigenfunction of $\lambda(V)$ normalized by $\|u\|_p = 1$; additionally $[u_{k_j}]_{s,p}^p \to [u]_{s,p}^p$. Thereby  $u_{k_j} \to u$ in	$\widetilde W^{s,p}(\Omega)$. In fact, proceeding as before, we observe that all subsequences of $\{u_k\}_{k\in\mathbb{N}}$ have a further subsequence that converges to $u$ in $\widetilde{W}^{s,p}(\Omega)$. From that, we conclude that $u_k\to u$  in $\widetilde{W}^{s,p}(\Omega).$
\end{rem}	
	
With the continuity of the functional $\lambda$ on hand, let us go further and prove a differentiability property. Recall that for $V \in L^q(\Omega)$ such that $\|V\|_{q;\Omega} =1$ the tangent space of $\partial B(0,1) = \{ V \in L^q(\Omega)\colon \|V\|_{q;\Omega} = 1\}$ at $V$ is  
\[
  T_V(\partial B(0,1))=\left\{ W \in L^q(\Omega) \colon \int_{\Omega} |V|^{q-2}V W\, \di x = 0\right\}.
\]
Given $W \in T_V(\partial B(0,1))$ and $\alpha\colon(-1,1)\to L^q(\Omega)$ a differentiable curve 	such that
\begin{gather*}
  \alpha(t) \in \partial B(0,1) \quad \text{for all } t\in(-1,1), \\
  \quad \alpha(0) = V \quad \text{ and } \quad \alpha'(0) = W,
\end{gather*}
we define $\tilde \lambda\colon (-1,1)\to \mathbb{R}$ by $\tilde \lambda(t)\coloneqq  \lambda(V_t),$ where $V_t=\alpha(t)$. By the previous lemma $\tilde \lambda$ is continuous. Moreover:
	
\begin{lem}\label{lem:derivable}
  $\tilde \lambda$ is differentiable at $t=0$ and
  \[
    \tilde \lambda'(0)=\int_{\Omega}W(x)|u|^p\, \di x,
  \]
  where $u$ is the positive eigenfunction associated to $\lambda(V)$ normalized by $\|u\|_p = 1$.
\end{lem}
\begin{proof}
  We begin the proof by observing that
  \[
	\tilde \lambda(t) - \tilde \lambda(0) =  \lambda(\alpha(t)) - \lambda(V) \le \int_{\Omega} (V_t(x)-V(x))|u|^p\, \di x
  \]
  then
  \begin{equation}\label{eq:der1}
    \begin{gathered}
				\limsup_{t \to 0^+} \dfrac{\tilde \lambda(t) - \tilde \lambda(0)}{t} \le \int_{\Omega} W(x)|u|^p\, \di x,\\
				\liminf_{t \to 0^-} \dfrac{\tilde \lambda(t) - \tilde \lambda(0)}{t} \ge \int_{\Omega} W(x)|u|^p\, \di x.		
    \end{gathered}
  \end{equation}
 
  Let $\{t_k\}_{k\in\mathbb{N}}$ be a sequence in $(0,1)$ such that $t_k\to 0^+$ and
  \[
    \lim_{k\to \infty}\dfrac{\tilde \lambda(t_k) - \tilde \lambda(0)}t = \liminf_{t\to0^+}\dfrac{\tilde \lambda(t) - \tilde \lambda(0)}t
  \]
  Since $\tilde \lambda(t_k) \to \tilde \lambda(0),$ by Remark \ref{rem:conv}, we have that
  \[
    u_k\to u \text{ in } \widetilde{W}^{s,p}(\Omega),
  \]
  where $u_k$ and $u$ are the positive normalized eigenfunctions associated to $\lambda(V_{t_k})$ and $\lambda(V),$ respectively. Then
  \begin{equation}\label{eq:der2}
    \begin{aligned}
	  \liminf_{t \to 0^+}\dfrac{\tilde \lambda(t) - \tilde \lambda(0)}{t}
				&=\lim_{k\to \infty}\dfrac{\tilde \lambda(t_k) - \tilde \lambda(0)}{t} \\
				&\ge \lim_{k\to \infty}\int_{\Omega} \dfrac{(V_{t_k}(x)-V(x))}{t_{k}} |u_{t_k}|^p \, \di x \\
				&=\int_{\Omega} W(x)|u|^p\, \di x.
	 \end{aligned}
  \end{equation}
  Similarly, we can see that
  \begin{equation}\label{eq:der3}
    \limsup_{t\to0^{-}}\dfrac{\tilde \lambda(t) - \tilde \lambda(0)}{t} \le \int_{\Omega} W(x)|u|^p\, \di x.
  \end{equation}
		
  Putting together \eqref{eq:der1}, \eqref{eq:der2} and \eqref{eq:der3}, we conclude that
  \[
    \lim_{t\to0}\dfrac{\tilde \lambda(t) - \tilde \lambda(0)}t = \int_{\Omega} W(x)|u|^p\, \di x,
  \]
  as we wanted to show.
  \end{proof}
	

\section{The Optimization problems}\label{optimizacion}

In this section we prove the existence and characterizations of optimal potentials for the first eigenvalue of \eqref{eq:autovalores}. As in the previous section, $\Omega \subset \R^n$ denotes  a bounded extension domain and $V$ is a function in $L^q(\Omega)$, with $q \in (1,\infty) \cap (\tfrac{n}{sp},\infty)$.

Let us begin with the optimization problem when the potential function $V$ is restricted to a bounded closed convex subset of $L^q(\Omega)$.

\begin{thm}\label{thm:maximo}
  Let $\mathcal{C}$ be a bounded closed convex subset of $L^q(\Omega)$. Then there exist a unique $V^* \in \C$ such that
  \[
    \lambda(V^*) = \max \{ \lambda (V) \colon V \in \C \}
  \]
  and $V_* \in \C$ (not necessarily unique) such that 
  \begin{equation}\label{eq:elmin}
  \lambda(V_*) = \min \{ \lambda(V) \colon V \in \C \}.
  \end{equation}
\end{thm}
\begin{proof}
  First we show that there is a unique $V^* \in \C$ such that
  \[
    \lambda(V^*) = \max \{ \lambda (V) \colon V \in \C \}.
  \] 
    Let $\{V_k\}_{k \in \N}\subset \C$ be such that
  \[
    \lim_{k \to \infty} \lambda(V_k) = \sup \{ \lambda(V) \colon V \in \C \}.
  \]
	
  Since $\mathcal{C}$ is bounded, there exist a subsequence (still denoted by $\{V_k\}_{k \in \N}$) and $V^* \in L^q(\Omega)$ such that
  \begin{equation}\label{eq:auxmax1}
	 V_k\rightharpoonup V^* \text{ weakly in } 	L^q(\Omega).
  \end{equation} 
In fact, since $\mathcal{C}$ is closed convex subset of $L^q(\Omega)$ it follows that $\C$ is weakly closed and so $V^*\in\mathcal{C}$. Then
  \begin{equation}\label{eq:auxmax2}
	 \lambda(V^*) \le \sup \{ \lambda(V) \colon V \in \C \}.
  \end{equation} 
		
  On the other hand, for any $\varepsilon>0$ there exists $u\in\widetilde{W}^{s,p}(\Omega)$ such that
  \[
    \lambda(V^*) + \varepsilon \ge 
    \dfrac12[u]_{s,p}^p + \int_{\Omega} V^*(x) |u|^p\, \di x.
  \]
  Then, using that $|u|^p\in L^{q'}(\Omega)$ (since $q > \tfrac{n}{sp}$) and \eqref{eq:auxmax1}, we deduce that
  \begin{align*}
    \lambda(V^*) + \varepsilon &\ge \dfrac12
    [u]_{s,p}^p + \int_{\Omega} V^*(x) |u|^p\, \di x \\
			&= \dfrac12[u]_{s,p}^p + \lim_{k \to \infty} \int_{\Omega} V_k(x) |u|^p\, \di x \\
			&\ge \lim_{k \to \infty} \lambda(V_k) \\
			&= \sup \{ \lambda(V) \colon V \in \C \}.
  \end{align*}
  Therefore, 
  \begin{equation}\label{eq:auxmax3}
    \lambda(V^*) \ge \sup \{ \lambda(V) \colon V \in \C \}.
  \end{equation} 
  The previous equation and \eqref{eq:auxmax2} imply
  \[
    \lambda(V^*) = \max \{ \lambda(V) \colon V \in \C \}.
  \]
		
  Suppose now that there exist $V_1,V_2 \in \C$ such that 
  \begin{equation}\label{eq:auxmax4}
    \lambda(V_1) = \lambda(V_2) = \max \{ \lambda(V) \colon V \in \C \}.
  \end{equation} 
  Since $\C$ is convex, we have that $V_3 = \dfrac{V_1+V_2}2 \in \C$. Moreover, since $\lambda$ is concave and \eqref{eq:auxmax4},
  \[
    \lambda (V_3) \ge \dfrac{\lambda(V_1) + \lambda(V_2)}{2} = \max \{ \lambda(V) \colon V \in \C \}
  \]
  Then
  \begin{equation}\label{eq:auxmax5}
    \lambda \left( \dfrac{V_1 + V_2}{2} \right) = \lambda(V_1) = \lambda(V_2) = \max \{ \lambda(V) \colon V \in \C \}.
  \end{equation}
		
  On the other hand, by Remark \ref{remark:unicidad}, there exist $u_1,u_2,u_3\in\widetilde{W}^{s,p}(\Omega)$ such that $u_i$ is the unique positive eigenfunction associated to $\lambda(V_i)$ normalized by $\|u_i\|_p = 1$, $i=1,2,3$. We claim that $u_1=u_2=u_3$. Suppose by contradiction that $u_1 \neq u_3$ or $u_2\neq u_3$. Then
  \begin{align*}
    \lambda(V_3) &= [u_3]_{s,p}^p + \int_{\Omega} \dfrac{V_1(x)+V_2(x)}2 |u_3|^p\, \di x \\
			&= \dfrac12 \left( \dfrac12[u_3]_{s,p}^p + 
			\int_{\Omega} V_1(x) |u_3|^p\, \di x +\dfrac12
			[u_3]_{s,p}^p + \int_{\Omega} V_2(x) |u_3|^p\, \di x \right) \\
			&> \dfrac{\lambda(V_1) + \lambda(V_2)}{2} \\
			&= \max \{ \lambda(V) \colon V \in \C \},
  \end{align*}
  which contradicts \eqref{eq:auxmax5}.
		
  Therefore,
  \[
    \h (u_1,v) + \int_{\Omega} V_1(x)|u_1|^{p-2}u_1 v\, \di x = \h (u_1,v) + \int_{\Omega} V_2(x)|u_1|^{p-2}u_1 v\, \di x
  \]
  for all $v \in \widetilde W^{s,p}(\Omega)$, that is
  \[
    \int_{\Omega}(V_1(x)-V_2(x))|u_1|^{p-2}u_1 v\, \di x = 0
  \]
  for all $v\in\widetilde{W}^{s,p}(\Omega)$. Then $V_1=V_2$ a.e. in $\Omega$.
  
  Finally we show that there is $V_* \in \C$ such that
  \[
    \lambda(V_*) = \min \{ \lambda (V) \colon V \in \C \}.
  \]
   Let $\{V_k\}_{k \in \N}\subset \C$ be such that
  \[
    \lim_{k \to \infty} \lambda(V_k) = \inf 
    \{ \lambda(V) \colon V \in \C \}.
  \]
  As before, we have that there 
  exist a subsequence (still denoted by $\{V_k\}_{k \in \N}$) and 
  $V_* \in \C$ such that
  \begin{equation}\label{eq:auxmax6}
	 V_k\rightharpoonup V_* \text{ weakly in } 	L^q(\Omega).
  \end{equation}
  Then
  \begin{equation}\label{eq:auxmax7}
	\lambda(V_*) \ge \inf \{ \lambda(V) \colon V \in \C \}.
  \end{equation}
  Let $\{u_k\}_{k\in\N}\subset\widetilde W^{s,p}(\Omega)$ be such that
  $\|u_k\|_p=1$ and
  \[
  		\lambda(V_k)=\dfrac12[u_k]_{s,p}^p+\int_{\Omega}V_k(x)|u_k|^p\di x.
  \]
  Then, by \eqref{eq:auxmax7} and Lemma \ref{lem:caux}, 
  there exist positive constants $C$ and $D$ independent of $k$ such that
  \begin{align*}
    	C &\ge \lambda(V_{k})= \dfrac12
	  [u_{k}]_{s,p}^p + \int_{\Omega} V_{k}(x)|u_{k}|^p\, \di x \\
	  &\ge \dfrac12
	  [u_{k}]_{s,p}^p - \dfrac14[u_{k}]_{s,p}^p - D 
	  \|V_{k}\|_{q;\Omega} \|u_{k}\|_p^p.
  \end{align*}
  Therefore
  \[
    [u_{k}]_{s,p}^p \le 4(C+D\sup\{\|V\|_{q;\Omega}\colon
    V\in\C\})
  \]
  for all $k \in \N$. Then, $\{u_{k}\}_{k \in \N}$ 
  is bounded in $\widetilde{W}^{s,p}(\Omega)$ and there exist a 
  subsequence (still denoted by $\{u_{k}\}_{k\in \N}$) and  
  $u \in \widetilde W^{s,p}(\Omega)$ such that
  \begin{gather*}
    	u_{k}\rightharpoonup u 
    	\text{ weakly in } \widetilde{W}^{s,p}(\Omega),\\
    	u_{k}\to u \text{ strongly in } L^{p}(\R^n),\\
		u_{k}\to u \text{ strongly in } L^{pq'}(\R^n).
  \end{gather*}
  Then, $\|u\|_{p}=1$ and, using \eqref{eq:auxmax6}, we have that
  \begin{align*}
		\lambda(V_*) &\ge \inf \{ \lambda(V) \colon V \in \C \}
		=\lim_{k \to \infty} \lambda(V_k) =\lim_{k \to \infty} \dfrac12
	  [u_{k}]_{s,p}^p + \int_{\Omega} V_{k}(x)|u_{k}|^p\, \di x \\
	  &\ge \dfrac12
	  [u]_{s,p}^p + \int_{\Omega} V(x)|u|^p\, \di x \ge \lambda(V_*).
\end{align*}  
\end{proof}

The next result is a characterization of the minimal potential $V_*$.

\begin{lem}\label{lema:minimizante}
  Let $u_*$ be the positive eigenfunction associated to $\lambda(V_*)$ such that $\| u_* \|_p = 1$. Then $V_*$ is the unique minimizer of the linear operator
  \[
    L(V) \coloneqq \int_{\Omega} V(x) |u_*|^p \, \di x
  \]
		relative to $V\in\mathcal{C}.$  
\end{lem}
\begin{proof}
  We first prove that $V_*$ is a minimizer. By \eqref{eq:elmin}, we have that
  \[
    \dfrac12[u_*]_{s,p}^p + \int_{\Omega} V_*(x) |u_*|^p\, \di x 
    = \lambda(V_*) 
			\le \lambda(V) 
			\le \dfrac12[u_*]_{s,p}^p + \int_{\Omega}V(x)|u_*|^p\, \di x 
  \]
  for all $V \in \C$. Therefore
  \[
    \int_{\Omega} V_*(x) |u_*|^p\, \di x \le \int_{\Omega} V(x)|u_*|^p\, \di x \quad \text{for all } V \in \C. 
  \]
		
  To prove the uniqueness, let $W \in \C$ such that 
  \[
    \int_{\Omega} W(x)|u_*|^p\, \di x = \min \left \{ L(V) \colon V \in \C \right \} = \int_{\Omega} V_*(x) |u_*|^p\, \di x.
  \]
  Then
  \[
    \lambda(V_*)
    = \dfrac12[u_*]_{s,p}^p + \int_{\Omega} V_*(x) |u_*|^p\, \di x
    = \dfrac12[u_*]_{s,p}^p + \int_{\Omega} W(x) |u_*|^p\, \di x 
    \ge \lambda(W).
  \]  
  Thus, by \eqref{eq:elmin}, $\lambda(V_*) = \lambda(W)$ and therefore $u_*$ is an eigenfunction associated to $\lambda(W)$. Then
  \[
    \int_{\Omega}(V_*(x)-W(x))|u_*|^{p-2} u_* v \, \di x = 0 
  \]
  for all $v \in \widetilde W^{s,p}(\Omega)$. Since $u_*>0$ in $\Omega$, we conclude that $V_* = W$ a.e. in $\Omega$.
	\end{proof}

\subsection{Optimization problems in a closed ball}

Let us now consider the case $\mathcal{C} = \bar B(0,1) \coloneqq \{ V \in L^q(\Omega) \colon \|V\|_{q;\Omega} \le 1 \}$, the unit closed ball in $L^q(\Omega)$. In this setting further characterizations of the extremal potentials can be provided.

Indeed, by Theorem~\ref{thm:maximo}, there exists a unique $V^* \in \bar B(0,1)$ such that 
  \begin{align*}
    \max \{ \lambda(V) \colon V \in \bar B(0,1) \} &= \lambda(V^*) 
    \le \dfrac12[u]_{s,p}^p + \int_{\Omega} V^*(x)|u|^p \, \di x \\
			&\le \dfrac12[u]_{s,p}^p + \int_{\Omega} \frac{|V^*(x)|}{\|V^*\|_{q;\Omega}}|u|^p \, \di x
  \end{align*}
for all $u \in \widetilde W^{s,p}(\Omega)$. Then
  \[
    \max \{ \lambda(V) \colon V \in \bar B(0,1) \} = \lambda(V^*) \le \lambda \left( \dfrac{|V^*|}{\|V^*\|_{q;\Omega}}\right).
  \]
Since $\tfrac{|V^*|}{\|V^*\|_{q;\Omega}} \in \partial B(0,1)$, then, by Theorem~\ref{thm:maximo}, $V^*$ is nonnegative  and $V^* \in \partial B(0,1)$. Moreover, by Lemma~\ref{lem:derivable}, we have that
  \[
    \int_{\Omega} W(x)|u^*|\, \di x = 0 \quad \text{for all } W \in T_{V^*}(\partial B(0,1)),
  \]
where $u^*$ is the positive eigenfunction of $\lambda(V^*)$ normalized by $\|u^*\|_p = 1$.
	
  This procedure proves the validity of the following result.
  \begin{thm}\label{teo:maxbola}
    Let $V^* \in \bar B(0,1)$ be the unique potential that satisfies
      \[
	    \lambda(V^*) = \max \{ \lambda(V) \colon V \in \bar B(0,1) \},
	  \]
    according to Theorem~\ref{thm:maximo}. Then $V^*$ is nonnegative, $V^* \in \partial B(0,1)$ and
	  \[
		\int_{\Omega} W(x)|u^*|\, \di x = 0 \quad \text{for all } W \in T_{V^*}(\partial B(0,1)),
	  \]
    where $u^*$ is the positive eigenfunction of $\lambda(V^*)$ normalized by $\|u^*\|_p = 1$.     
  \end{thm}
	
Similarly, we have that
  \begin{thm}\label{teo:minbola}
    There exists $V_* \in \partial B(0,1)$ such that
      \[
		\lambda(V_*) = \min \{\lambda(V) \colon V \in \bar B(0,1) \}.
	  \]
	Moreover, $V_*$ is nonpositive, $\|V_*\|_{q;\Omega} = 1$ and     
	  \[
		\int_{\Omega} W(x)|u_*|\, \di x = 0 \quad 
		\text{for all } W \in T_{V_*}(\partial B(0,1)),
	  \]
	where $u_*$ is the positive eigenfunction of $\lambda(V_*)$ normalized by $\|u_*\|_p = 1$.      
  \end{thm}
	
  \begin{cor}
    In the notation of Theorem~\ref{teo:maxbola} and \ref{teo:minbola}, we have $\Omega=\supp (V^*)=\supp (V_*)$ and there exist two constants $C^*$ and $C_*$ such that
      \begin{gather*}
	    |u^*(x)|^p = C^* |V^*(x)|^{q-1}, \\
	    |u_*(x)|^p = C_* |V_*(x)|^{q-1}, 
	  \end{gather*}
	for a.e. $x \in \Omega$.
  \end{cor}
  \begin{proof}
	See the proofs of Proposition~3.10 and Theorem~3.11 in \cite{FBDP}.
  \end{proof}

\subsection{Optimization problems in the class of rearrangements of a given potential} 

Let $V_0 \in L^q(\Omega)$ and $\mathcal{R}(V_0)$ be the set of rearrangements of $V_0$, that is $V \in \mathcal{R}(V_0)$ iff $V\colon \Omega\to\mathbb{R}$ is a measurable function and 
  \[
    |\{x\in\Omega\colon V(x)\ge t\}|=|\{x\in\Omega\colon V_0(x) \ge t\}|
  \]  
for any $t \in \R$.	
	
  \begin{rem}\label{rem:r1} 
    If $V \in \mathcal{R}(V_0)$ then $V \in L^q(\Omega)$ and $\|V\|_{q;\Omega} = \|V_0\|_{q;\Omega}$. See, for instance, \cite[Lemma 2.1]{MR870963}.	
  \end{rem}

Let $\overline{\mathcal{R}(V_0)}$ be the the weak closure of $\mathcal{R}(V_0)$.  In \cite[Theorem~6]{MR870963}, the author proves that $\overline{\mathcal{R}(V_0)}$ is convex, see also \cite{MR0192552,MR0234263}. Hence $\overline{\mathcal{R}(V_0)}$ is strongly closed. Then, by Remark \ref{rem:r1}, we have that $\overline{\mathcal{R}(V_0)}$ is a bounded closed convex subset of $L^q(\Omega)$.

Thus, by Theorems~\ref{thm:maximo}, we have that
  \begin{itemize}
    \item There exists a unique $V^* \in \overline{\mathcal{R}(V_0)}$ so that
		  \[
		    \lambda(V^*) = \max \{ \lambda(V) \colon V \in \overline{\mathcal{R}(V_0)} \};
		  \]
    \item There exists $V_* \in \overline{\mathcal{R}(V_0)}$ so that
		  \begin{equation}\label{eq:elmin1}
		    \lambda(V_*) = \min \{ \lambda(V) \colon V \in \overline{\mathcal{R}(V_0)} \}.
		  \end{equation}
  \end{itemize}

By \cite[Theorems~1 and 4]{MR870963}, there is $W\in\mathcal{R}(V_0)$ so that
  \[
    L(W)=\min\{L(V)\colon V\in\mathcal{R}(V_0)\}.
  \]
Then, by Lemma \ref{lema:minimizante}, we have that $V_* = W$ a.e in $\Omega$. Hence $V_* \in \mathcal{R}(V_0)$. Moreover, by Lemma \ref{lema:minimizante} and Theorem~5 in \cite{MR870963}, there is a decreasing function $\varphi\colon\mathbb{R}\to\mathbb{R}$ so that 	$V_* = \varphi\circ |u_*|^p$. Therefore we proved the next result.
	
  \begin{thm}\label{teo:minreo}
    Let $V_0 \in L^q(\Omega).$ There is a rearrangement $V_*$ of  $V_0$ in $\Omega$ such that 
      \[
        \lambda(V_*) = \min \{ \lambda(V) \colon V \in \mathcal{R}(V_0) \}.
      \]
    Moreover there exists a decreasing function $\varphi \colon \R \to \mathbb{R}$ so that $V_* = \varphi \circ |u_*|^p,$ where $u_*$ is the positive eigenfunction associated to $\lambda(V_*)$ such that $\|u_*\|_p = 1$.
	\end{thm}

%
%
%
%

\appendix

\section{Regularity of fractional p-eigenfunctions}

We begin by proving that the eigenfunctions are bounded.

  \begin{lem}\label{lema:cotainfty}
    Let $\Omega\subset\mathbb{R}^n$ be a bounded extension domain
     and $V\in L^q(\Omega)$ with $q \in (1,\infty) \cap (\tfrac{n}{sp},\infty)$.  If $u$ is an eigenfunction associated to $\lambda$ then $u\in L^{\infty}(\mathbb{R}^n)$.
  \end{lem}
  \begin{proof}
    In this proof we follow ideas from \cite{FP}.
	
	If $ps>n,$ by Theorem~\ref{thm:teoemb}, then the assertion holds.
	Then let us suppose that $sp \le n$. We will show that if $\|u_+\|_{pq'} \le \delta$ then $u_+$ is bounded, where $\delta > 0$ must be determined.
	
	For $k\in\mathbb{N}_0$ we define the function $u_k$ by
	\[
	  u_k\coloneqq(u-1+2^{-k})_+.
	\]
	Observe that, $u_0= u_+$ and for any $k\in\mathbb{N}_0$ we have that $u_k\in \widetilde{W}^{s,p}(\Omega)$, 
	  \begin{equation}\label{eq:lci1}
	    \begin{aligned}
	      &u_{k+1}\le u_k \text{ a.e. } \mathbb{R}^n,\\
	      &u <(2^{k+1}-1)u_k\text{ in } \{u_{k+1}>0\},\\
	      &\{u_{k+1}>0\}\subset\{u_k>2^{-(k+1)}\}.
	    \end{aligned}
      \end{equation}
	Now, since
	\[
	|v_+(x)-v_+(y)|^p\le |v(x)-v(y)|^{p-2}(v(x)-v(y))(v_+(x)-v_+(y))
	\quad\forall x,y\in\mathbb{R}^n,
	\]
	for any function $v \colon \R^n \to \R$,  we have that
	  \begin{align*}
	    \dfrac12[u_{k+1}]_{s,p}^p &\le \h(u,u_{k+1}) \\
                          &= \lambda \int_{\Omega} u^{p-1} w_{k+1} \, \di x - \int_{\Omega} V(x)u^{p-1}u_{k+1}\, \di x \\
	                      &\le |\lambda| \int_{\Omega} u^{p-1} w_{k+1} \, \di x + \int_{\Omega} V_-(x) u^{p-1}u_{k+1}\, \di x,
	  \end{align*}
	for all $k\in\mathbb{N}_0$. Then, by \eqref{eq:lci1} and H\"older inequality, we have that
	  \begin{equation}\label{eq:lci2}
	    \begin{aligned}
	        \dfrac12[u_{k+1}]_{s,p}^p &\le |\lambda|\int_{\Omega} u^{p-1} w_{k+1} \, \di x + \int_{\Omega} V_{-}(x)u^{p-1}u_{k+1}\, \di x \\
            &\le (2^{k+1}-1)^{p-1}\left(|\lambda| \|u_k\|_p^p + \int_{\Omega}V_{-}(x)u_k^p\, \di x \right) \\
	        &\le(2^{k+1}-1)^{p-1}\left(|\lambda| |\Omega|^{\tfrac{1}q} + \|V\|_{q;\Omega} \right) \|u_k\|_{pq'}^p
	    \end{aligned}
	  \end{equation}
    for all $k\in\mathbb{N}_0$.
	
	On the other hand, in the case $sp<n,$ using H\"older's inequality, Theorem~\ref{thm:teoembcont}, \eqref{eq:lci1}, and 	Chebyshev's inequality, for any $k\in\mathbb{N}_0$ we have that 
	  \begin{equation}\label{eq:lci3}
	    \begin{aligned}
	      \|u_{k+1}\|_{pq'}^p &\le \|u_{k+1}\|_{p_s^*}^p |\{u_{k+1}>0\}|^{\tfrac{1}{q^{\prime}} - \tfrac{p}{p_s^*}} \\
          &\le C[u_{k+1}]_{s,p}^p |\{u_{k+1}>0\}|^{\tfrac{sp}{n}-\tfrac{1}q} \\
	      &\le C[u_{k+1}]_{s,p}^p |\{u_{k}>2^{-(k+1)}\}|^{\tfrac{sp}{n}-\tfrac{1}q} \\
	      &\le C[u_{k+1}]_{s,p}^p \left( 2^{(k+1)p} \|u_k\|_{pq'}^p \right)^{q^\prime(\tfrac{sp}{n}-\tfrac{1}q)},
	    \end{aligned}
	  \end{equation}
	where $C$ is a constant independent of $k$. Then, by \eqref{eq:lci2} and \eqref{eq:lci3}, for any $k\in\mathbb{N}_0$ we have that 
	  \begin{equation}\label{eq:lci4}
	    \begin{aligned}
	      \|u_{k+1}\|_{pq'}^p &\le C \left(2^{(k+1)p}\|u_k\|_{pq'}^p \right)^{1+\alpha},
	    \end{aligned}
	  \end{equation}
	where $C$ is a constant independent of $k$ and 	$\alpha=q^\prime(\tfrac{sp}{n}-\tfrac{1}q)>0$.
	
	Similarly, in the case $sp=n,$ taking $r>pq^\prime$ and proceeding as in the previous case $sp<n$ (with $r$ in place of $p_s^*$), we have that \eqref{eq:lci4} holds with $\alpha=1-\tfrac{pq^\prime}r>0$.
	
	Therefore if $sp\le n$ then there exist $\alpha>0$ and a constant $C>1$  such that
	  \begin{equation}\label{eq:lci5}
	    \|u_{k+1}\|_{pq'}^p \le C^k \left(\|u_k\|^p_{pq'} \right)^{1+\alpha},
	\end{equation}
	for any $k\in\mathbb{N}_0$. Hence, if $\|u_0\|_{pq'}^p = \|u_+\|_{pq'}^p \le C^{\tfrac{-1}{\alpha^2}} \eqqcolon\delta^p$ then $u_{k}\to0$ in $L^{pq^\prime}(\Omega)$. On the oher hand $u_k\to(u-1)_+$ a.e in $\mathbb{R}^n,$ then $(u-1)_+\equiv0$ in $\mathbb{R}^n.$ Therefore $u_+$ is bounded. 
	
	Finally, taking $-u$ in place of $u$ we have  that $u_-$ is bounded if $\|u_-\|_{pq'} < \delta$. Therefore $u$ is bounded. 
\end{proof}

Finally we show a regularity result.

\begin{thm}\label{teo:regularidad}
	Let $\Omega\subset\mathbb{R}^n$ be a bounded extension domain, and $V\in L^\infty(\Omega).$ If $u$ is an eigenfunction
	associated to $\lambda$ then there is $\alpha\in(0,1)$
	such that $u\in C^{\alpha}(\overline{\Omega}).$
\end{thm}

\begin{proof}
	By Lemma \ref{lema:cotainfty}, we have that $u\in L^{\infty}(\Omega).$
	Then $(\lambda-V(x))|u|^{p-2}u\in L^{\infty}(\Omega).$
	Therefore, by \cite[Theorem~1.1]{IMS}, there is  $\alpha\in(0,1)$
	such that $u\in C^{\alpha}(\overline{\Omega}).$	
\end{proof}

\section*{Acknowledgements}
This paper was partially supported by Universidad de Buenos Aires under grant UBACyT 20020130100283BA and by ANPCyT under grant PICT 2012-0153. J. Fern\'andez Bonder and Leandro M. Del Pezzo are members of CONICET.

\bibliographystyle{plain}
\bibliography{BiblioFract}

\def\cprime{$'$}
\begin{thebibliography}{10}

\bibitem{Adams}
Robert~A. Adams.
\newblock {\em Sobolev spaces}.
\newblock Academic Press [A subsidiary of Harcourt Brace Jovanovich,
  Publishers], New York-London, 1975.
\newblock Pure and Applied Mathematics, Vol. 65.

\bibitem{Amghibech}
S.~Amghibech.
\newblock On the discrete version of {P}icone's identity.
\newblock {\em Discrete Appl. Math.}, 156(1):1--10, 2008.

\bibitem{Ashbaugh}
Mark~S. Ashbaugh and Evans~M. Harrell, II.
\newblock Maximal and minimal eigenvalues and their associated nonlinear
  equations.
\newblock {\em J. Math. Phys.}, 28(8):1770--1786, 1987.

\bibitem{BLP}
L.~Brasco, E.~Lindgren, and E.~Parini.
\newblock The fractional {C}heeger problem.
\newblock {\em Interfaces Free Bound.}, 16(3):419--458, 2014.

\bibitem{BF}
Lorenzo Brasco and Giovanni Franzina.
\newblock Convexity properties of {D}irichlet integrals and {P}icone-type
  inequalities.
\newblock {\em Kodai Math. J.}, 37(3):769--799, 2014.

\bibitem{MR0192552}
James~R. Brown.
\newblock Approximation theorems for {M}arkov operators.
\newblock {\em Pacific J. Math.}, 16:13--23, 1966.

\bibitem{MR870963}
G.~R. Burton.
\newblock Rearrangements of functions, maximization of convex functionals, and
  vortex rings.
\newblock {\em Math. Ann.}, 276(2):225--253, 1987.

\bibitem{DD}
Fran{\c{c}}oise Demengel and Gilbert Demengel.
\newblock {\em Functional spaces for the theory of elliptic partial
  differential equations}.
\newblock Universitext. Springer, London, 2012.
\newblock Translated from the 2007 French original by Reinie Ern{\'e}.

\bibitem{di2014local}
Agnese Di~Castro, Tuomo Kuusi, and Giampiero Palatucci.
\newblock Local behavior of fractional {$p$}-minimizers.
\newblock {\em Ann. Inst. H. Poincar\'e Anal. Non Lin\'eaire},
  33(5):1279--1299, 2016.

\bibitem{DNPV}
Eleonora Di~Nezza, Giampiero Palatucci, and Enrico Valdinoci.
\newblock Hitchhiker's guide to the fractional {S}obolev spaces.
\newblock {\em Bull. Sci. Math.}, 136(5):521--573, 2012.

\bibitem{FBDP}
Juli{\'a}n Fern{\'a}ndez~Bonder and Leandro~M. Del~Pezzo.
\newblock An optimization problem for the first eigenvalue of the
  {$p$}-{L}aplacian plus a potential.
\newblock {\em Commun. Pure Appl. Anal.}, 5(4):675--690, 2006.

\bibitem{FP}
Giovanni Franzina and Giampiero Palatucci.
\newblock Fractional $p-$eigenvalues.
\newblock {\em Riv. Math. Univ. Parma (N.S.)}, 5(2):315--328, 2014.

\bibitem{Grisvard}
P.~Grisvard.
\newblock {\em Elliptic problems in nonsmooth domains}, volume~24 of {\em
  Monographs and Studies in Mathematics}.
\newblock Pitman (Advanced Publishing Program), Boston, MA, 1985.

\bibitem{IMS}
Antonio Iannizzotto, Sunra Mosconi, and Marco Squassina.
\newblock Global {H}\"older regularity for the fractional {$p$}-{L}aplacian.
\newblock {\em Rev. Mat. Iberoam.}, 32(4):1353--1392, 2016.

\bibitem{KKL}
J.~{Korvenp{\"a}{\"a}}, T.~{Kuusi}, and E.~{Lindgren}.
\newblock {Equivalence of solutions to fractional $p$-Laplace type equations}.
\newblock {\em J. Math. Pures Appl, to appear}, 2016.

\bibitem{Laskin}
Nick Laskin.
\newblock Fractional {S}chr\"odinger equation.
\newblock {\em Phys. Rev. E (3)}, 66(5):056108, 7, 2002.

\bibitem{LL}
Erik Lindgren and Peter Lindqvist.
\newblock Fractional eigenvalues.
\newblock {\em Calc. Var. Partial Differential Equations}, 49(1-2):795--826,
  2014.

\bibitem{R}
R.~T. Rockafellar.
\newblock {\em Network flows and monotropic optimization}.
\newblock Pure and Applied Mathematics (New York). John Wiley \& Sons, Inc.,
  New York, 1984.
\newblock A Wiley-Interscience Publication.

\bibitem{MR0234263}
John~V. Ryff.
\newblock Majorized functions and measures.
\newblock {\em Nederl. Akad. Wetensch. Proc. Ser. A 71 = Indag. Math.},
  30:431--437, 1968.

\end{thebibliography}
\end{document}